%% file: main.tex
\documentclass[12pt]{article}
\DeclareUnicodeCharacter{FFFD}{}
\input{kankyo}

\title{\textbf{ $\mathfrak b_1$-Verma $\mathfrak b_2$-dual Verma supermodules
}}
\author{\textbf{Shunsuke Hirota}}

\date{\textit{\today}}

\begin{document}

\maketitle
\input{ab}
\tableofcontents

\section{Introduction}\label{sec:intro}

\input{intro}

\subsection{Acknowledgments}
This work was supported by the Japan Society for the Promotion of Science (JSPS) through the Research Fellowship for Young Scientists (DC1), Grant Number JP25KJ1664.

\section{Well known basics}\label{sec:basics}
In this section, we summarize basic facts about $\mathfrak{gl}(m|n)$ and its
representation theory. All material in this section is
 well known.

\input{0}

\subsection{Root systems}\label{subsec:roots}

\input{1}

\subsection{Finite Young lattices $L(m,n)$}\label{subsec:finite-young-lattices}
\input{2}

\subsection{Contragredient duality}\label{subsec:contragredient-duality}

\input{3}

\subsection{Verma modules}\label{subsec:verma-modules}

\input{4}

\subsection{Category $\mathcal O$}\label{subsec:category-o}
\input{12}

\subsection{Kac functors}\label{subsec:kac-functors}
\input{14}

\section{Odd reflection graph of Verma modules}\label{sec:odd-reflection-graph}
This section’s results hold, in principle, for any regular symmetrizable Kac-Moody Lie superalgebras \cite{serganova2011kac,bonfert2024weyl} and double bozonization of Nichols algebras of diagonal type \cite{vay2023linkage}.

\subsection{Edge contraction}\label{subsec:edge-contraction}
\input{25}

\subsection{Odd Verma's theorem}\label{subsec:odd-verma-theorem}

\input{26}

\section{Main results}\label{sec:main-results}

\subsection{Applying odd Verma's theorem}\label{subsec:applying-odd-verma-theorem}

\input{81}

\subsection{A bridge in odd reflection graphs}\label{subsec:bridge-odd-reflection-graphs}

\input{83}

\section{Further Study of odd reflection graphs}\label{sec:further-odd-graphs}
\subsection{Totally disconnected weights}\label{subsec:totally-disconnected-weights}

\input{92}

% 参考文献
\bibliographystyle{plainnat}  % natbib 用のスタイル
\bibliography{references}

\noindent
\textsc{Shunsuke Hirota} \\
\textsc{Department of Mathematics, Kyoto University} \\
Kitashirakawa Oiwake-cho, Sakyo-ku, 606-8502, Kyoto \\
\textit{E-mail address}: \href{mailto:hirota.shunsuke.48s@st.kyoto-u.ac.jp}{hirota.shunsuke.48s@st.kyoto-u.ac.jp}

\end{document}

%% file: kankyo.tex
\usepackage{amsmath,amssymb,amsthm}
\usepackage{ytableau}
\usepackage{graphicx}
\usepackage{enumitem}
\usepackage{appendix}
\usepackage{tikz}
\usepackage{makecell}
\usepackage[all]{xy}
\usepackage[numbers,sort&compress]{natbib}
\usepackage[a4paper,top=1in,bottom=1in,left=1.25in,right=1.25in]{geometry}

\setlist[enumerate,1]{leftmargin=*}
\setlist[itemize]{leftmargin=1em}

\usepackage{hyperref}
\usepackage[capitalize,nameinlink]{cleveref}
\usepackage{aliascnt}

\theoremstyle{plain}
\newtheorem{theorem}{Theorem}[section]

\newaliascnt{lemma}{theorem}
\newtheorem{lemma}[lemma]{Lemma}
\aliascntresetthe{lemma}

\newaliascnt{proposition}{theorem}
\newtheorem{proposition}[proposition]{Proposition}
\aliascntresetthe{proposition}

\newaliascnt{corollary}{theorem}
\newtheorem{corollary}[corollary]{Corollary}
\aliascntresetthe{corollary}

\theoremstyle{definition}
\newaliascnt{definition}{theorem}
\newtheorem{definition}[definition]{Definition}
\aliascntresetthe{definition}

\newaliascnt{example}{theorem}
\newtheorem{example}[example]{Example}
\aliascntresetthe{example}

\theoremstyle{remark}
\newaliascnt{remark}{theorem}
\newtheorem{remark}[remark]{Remark}
\aliascntresetthe{remark}

\newaliascnt{question}{theorem}

\aliascntresetthe{question}

\newaliascnt{conjecture}{theorem}

\aliascntresetthe{conjecture}

\crefname{theorem}{Theorem}{Theorems}       \Crefname{theorem}{Theorem}{Theorems}
\crefname{lemma}{Lemma}{Lemmas}             \Crefname{lemma}{Lemma}{Lemmas}
\crefname{proposition}{Proposition}{Propositions} \Crefname{proposition}{Proposition}{Propositions}
\crefname{corollary}{Corollary}{Corollaries} \Crefname{corollary}{Corollary}{Corollaries}
\crefname{definition}{Definition}{Definitions} \Crefname{definition}{Definition}{Definitions}
\crefname{example}{Example}{Examples}       \Crefname{example}{Example}{Examples}
\crefname{remark}{Remark}{Remarks}          \Crefname{remark}{Remark}{Remarks}
\crefname{question}{Question}{Questions}    \Crefname{question}{Question}{Questions}
\crefname{conjecture}{Conjecture}{Conjectures} \Crefname{conjecture}{Conjecture}{Conjectures}

\crefname{equation}{Equation}{Equations}    \Crefname{equation}{Equation}{Equations}
\crefname{figure}{Figure}{Figures}          \Crefname{figure}{Figure}{Figures}
\crefname{table}{Table}{Tables}             \Crefname{table}{Table}{Tables}
\renewenvironment{proof}{\noindent\textbf{Proof.}\ }{\hfill$\square$\par}

%% file: ab.tex
We show that if a module M over a basic classical Lie superalgebra of type type I is simultaneously a Verma module with respect to some Borel \(\mathfrak b_1\) and a dual Verma module with respect to  Borel \(\mathfrak b_2\), then M is isomorphic to a Verma module with respect to either distinguished or an anti-distinguished Borel. 
Our method proceeds by analyzing edge contractions of the finite Young lattice that controls the combinatorics of odd reflections. In principle, the same strategy, for the most part, applies to all basic classical Lie superalgebras.

%% file: intro.tex
Basic classical Lie superalgebras are widely regarded as a natural generalization of semisimple Lie algebras.  Unlike the semisimple Lie algebra case, basic classical Lie superalgebras admit genuinely different (non-conjugate) Borel subalgebras. Their root systems admit genuinely different choices of simple systems; the associated simple reflections come in two types—\emph{even} and \emph{odd}—and together they naturally form a groupoid. Since these are special cases of Weyl groupoids in the sense of Heckeberger and Yamane\cite{heckenberger2008generalization,heckenberger2020hopf,bonfert2024weyl}, and Weyl groupoids play a crucial role in the classification of Nichols algebras of diagonal type \cite{angiono2024root,andruskiewitsch2017finite}, it is natural and compelling to study these Lie superalgebra cases  and their attendant representation theory.

In representation theory  from this viewpoint, a central object is the \(\mathfrak b\)-Verma module with respect to a Borel subalgebra \(\mathfrak b\): for \(\lambda\in\mathfrak h^{*}\),
\(
M^{\mathfrak b}(\lambda):=\operatorname{Ind}_{\mathfrak b}^{\mathfrak g}\,k_{\lambda},
\)
which has simple top \(L^{\mathfrak b}(\lambda)\).
For each Borel subalgebra \(\mathfrak b\), there is a BGG category \(\mathcal O\) that serves as the natural home for \(\mathfrak b\)-Verma modules; it has been extensively studied\cite{coulembier2016primitive,coulembier2017gorenstein,coulembier2017homological,mazorchuk2014parabolic,musson2012lie,serganova2017representations}, in particular for \(\mathfrak{gl}(m|n)\)\cite{brundan2014representations,brundan2017tensor,brundan2019whittaker,chen2023some,cheng2013equivalence,cheng2015brundan,hoyt2019integrable}.
This category \(\mathcal O\) depends only on the even part \(\mathfrak b_{\bar 0}\), so \(\mathcal O\) mixes Verma modules attached to essentially different Borels. In particular, for each \(\mathfrak b\) the category \(\mathcal O\) carries a different highest weight structure.

The aim of this paper is to analyze, in the prototypical case of the general linear Lie superalgebra \(\mathfrak{gl}(m|n)\), the relationship between changes of Borel subalgebras and representation theory. In particular, we answer the basic question of when the contragredient dual of a Verma module is again a Verma module. In the classical (non-super) BGG category \(\mathcal O\), this is well known: it happens exactly when the Verma module is antidominant and  simple. Our generalization to \(\mathfrak{gl}(m|n)\) setting is summarized below.

\begin{theorem}\label{b1b2}
Let $\mathfrak g=\mathfrak{gl}(m|n)$ and \(\lambda\) be an integral weight and let \(\mathfrak b_1\) be a Borel subalgebra.
If \(M^{\mathfrak b_1}(\lambda)\cong M^{\mathfrak b_2}(\mu)^{\vee}\) for some \(\mu\) and some Borel \(\mathfrak b_2\) , then \(M^{\mathfrak b_1}(\lambda)\) is isomorphic to  a Verma module with respect to either distinguished Borel \(()\) or anti-distinguished Borel \((n^m)\), and \(L^{\mathfrak b_1}(\lambda)\) is antidominant.
Conversely, if \(\lambda\) is antidominant, then \(M^{()}(\lambda)\cong  M^{(n^{m})}(\lambda-2\rho_{\bar 1})^{\vee}\).
\end{theorem}

Note that an atypical Verma module is always not simple.

Our main technical work toward \cref{b1b2} is a detailed analysis, for \(\mathfrak{gl}(m|n)\), of the finite Young lattice \(L(m,n)\) and of its appropriate edge contractions, as illustrated in the examples below. The finite Young lattice \(L(m,n)\) is a well-studied object in combinatorics in its own right\cite{coggins2024visual,stanley2012topics}.

\begin{example}\label{l32}

Let $\mathfrak g=\mathfrak{gl}(3|2)$ .
        \(L(3,2)\) is the following edge-colored graph.

\begin{center}
    \begin{tikzpicture}[scale=1.5]
    % Nodes with Young diagrams in French notation and adjusted coordinates
     \node (A) at (0, 0) {\(\emptyset\)};
    \node (B) at (1, 0) {\(\begin{ytableau} ~ \end{ytableau}\)};

    \node (C) at (2.5, 0.75) {\(\begin{ytableau} ~ & ~ \\ \end{ytableau}\)};
    \node (D) at (4, 0.75) {\(\begin{ytableau} ~ \\ ~ & ~ \end{ytableau}\)};
    \node (E) at (5.5, 0.75) {\(\begin{ytableau} ~ & ~ \\ ~ & ~ \end{ytableau}\)};
    \node (F) at (7, 0) {\(\begin{ytableau} ~ \\ ~ & ~ \\ ~ & ~  \end{ytableau}\)};
        \node (G) at (9, 0) {\(\begin{ytableau} ~ & ~ \\  ~ & ~  \\ ~ & ~ \end{ytableau}\)};
    \node (H) at (2.5, -0.75) {\(\begin{ytableau} ~ \\ ~ \end{ytableau}\)};
    \node (I) at (4, -0.75) {\(\begin{ytableau} ~ \\ ~ \\ ~ \end{ytableau}\)};
    \node (J) at (5.5, -0.75) {\(\begin{ytableau} ~ \\ ~ \\ ~ & ~ \end{ytableau}\)};
    
 % Edges with labels
\draw (A) -- node[above] {(1,1)} (B);
\draw (B) -- node[above] {(2,1)} (C);
\draw (C) -- node[above] {(1,2)} (D);
\draw (D) -- node[above] {(2,2)} (E);
\draw (E) -- node[above] {(1,3)} (F);
\draw (F) -- node[above] {(2,3)} (G);
\draw (B) -- node[above] {(1,2)} (H);
\draw (D) -- node[above] {(2,1)} (H);
\draw (D) -- node[above] {(1,3)} (J);
\draw (F) -- node[above] {(2,2)} (J);
\draw (H) -- node[above] {(1,3)} (I);
\draw (I) -- node[above] {(2,1)} (J);
\end{tikzpicture}.
\end{center}

Let $\lambda=\varepsilon_1-\delta_2$.
Then the edge contraction is \(
L(3,2)_\lambda \;\cong\; L(3,2)\big/\{(1,1),(1,2),(2,3)\}.
\)

The resulting quotient is:
\begin{center}
\begin{tikzpicture}[xscale=1.35]
  % Nodes (coordinates scaled by 1.25)
  \node (A) at (0,1.875) {$\left\{\, \varnothing,\ \begin{ytableau} ~ \end{ytableau},\ \begin{ytableau} ~ \\ ~ \end{ytableau} \,\right\}$};
  \node (B) at (3.2, 3.2)
  {$\left\{\,  \begin{ytableau}
     
      ~ & ~
    \end{ytableau},\ \ \begin{ytableau}
      ~\\
      ~ & ~
    \end{ytableau} \,\right\}$};
  \node (C) at (5.8, 3.2) {
    \begin{ytableau}
      ~ & ~\\
      ~ & ~
    \end{ytableau}
  };
  \node (D) at (8.4375,1.875)
  {$\left\{\,
    \begin{ytableau}
      ~ \\
      ~ & ~\\
      ~ & ~
    \end{ytableau}
    \,,\ 
    \begin{ytableau}
      ~ & ~\\
      ~ & ~\\
      ~ & ~
    \end{ytableau}
  \,\right\}$};

  \node (E) at (2.8125, 0.9375) {
    \begin{ytableau}
      ~\\
      ~\\
      ~
    \end{ytableau}
  };
  \node (F) at (5.625, 0.9375) {
    \begin{ytableau}
      ~\\
      ~\\
      ~ & ~
    \end{ytableau}
  };

  % Edges (remaining colors)
  \draw (A) -- node[above] {(2,1)} (B);
  \draw (B) -- node[above] {(2,2)} (C);
  \draw (A) -- node[above] {(1,3)} (E);
  \draw (B) -- node[above] {(1,3)} (F);
  \draw (C) -- node[above] {(1,3)} (D);
  \draw (E) -- node[above] {(2,1)} (F);
  \draw (D) -- node[above] {(2,2)} (F);
\end{tikzpicture}
\end{center}
\end{example}

In the finite Young lattice \(L(m,n)\), vertices correspond to all Borel subalgebras of \(\mathfrak{gl}(m|n)\) containing the fixed even Borel of \(\mathfrak g_{\bar0}\), and edges correspond to odd reflections between such Borels (coloreded by the reflected isotropic odd simple root.) Furthermore, for a fixed weight \(\lambda\) we consider the edge contraction \(L(m,n)\rightsquigarrow L(m,n)_{\lambda}\) (we call it odd reflection graph of  Verma modules \(\{\,M^{\mathfrak b}(\lambda-\rho^{\mathfrak b})\,\}_{\mathfrak b}\) , here,  \(\rho^{\mathfrak b}\) is \({\mathfrak b}\)-Weyl vector ).
The vertex set of \(L(m,n)_{\lambda}\) can be identified with the isomorphism classes of Verma modules having the same character,
\(\{\,M^{\mathfrak b}(\lambda-\rho^{\mathfrak b})\,\}_{\mathfrak b}\),
and walks in this graph correspond to compositions of  homomorphisms among these modules.
In prior work \cite{hirota2025odd}, these homomorphisms are described; \cref{b1b2} follow by applying that description.
Note that, for the type~I Lie superalgebra \(\mathfrak{osp}(2\mid 2n)\), the corresponding graph is always a line segment, so \cref{b1b2} is immediate.
In \cref{sec:further-odd-graphs}, we reinterpret the class of \emph{totally disconnected} dominant weights (as introduced in Su and Zhang\cite{su2007character} ,see also \cite{chmutov2014weyl}) in terms of changes of Borel.

%% file: 0.tex
Let the base field \(k\) be an algebraically closed field of characteristic \(0\).
Let $\mathbf{Vec}$ denote the category of vector spaces, and $\mathbf{sVec}$ the category of supervector spaces.
Let $\Pi:\mathbf{sVec}\to\mathbf{sVec}$ be the \emph{parity–shift functor}.
Let $F:\mathbf{sVec}\to\mathbf{Vec}$ be the monoidal functor that forgets the $\mathbb{Z}/2\mathbb{Z}$–grading.
Note that $F$ is \emph{not} a symmetric monoidal functor.
In what follows, whenever we refer to the dimension of a supervector space,
we mean this total (ordinary) dimension:
\[
\dim V := \dim_k F(V).
\]

From now on, we will denote by \( \mathfrak{g} \) a finite dimensional Lie superalgebra.
We consider the category \( \mathfrak{g}\text{-sMod} \), where morphisms respects \(\mathbb{Z}/2\mathbb{Z}\)-grading. (This is the module category of a monoid object in \( \mathfrak{g}\text{-sMod} \) in the sense of \cite{etingof2015tensor}.)

%% file: 1.tex
\label{glmndef}
Let \( V = V_{\overline{0}} \oplus V_{\overline{1}} \) be a \( \mathbb{Z}/2\mathbb{Z} \)-graded vector space, where \( V_{\overline{0}} \) (the even part) is spanned by \( v_1, \dots, v_m \) and \( V_{\overline{1}} \) (the odd part) is spanned by \( v_{m+1}, \dots, v_{m+n} \).

The space \( \operatorname{End}(V) \) is spanned by basis elements \( E_{ij} \), defined by:  \(
E_{ij} \cdot v_k = \delta_{jk} v_i.
\)

The general linear Lie superalgebra \( \mathfrak{gl}(m|n) \) is defined as the Lie superalgebra spanned by all \( E_{ij} \) with \( 1 \leq i, j \leq m+n \), under the supercommutator:
\(
[E_{ij}, E_{kl}] = E_{ij} E_{kl} - (-1)^{|E_{ij}| |E_{kl}|} E_{kl} E_{ij},
\))
where \( |E_{ij}| = \overline{0} \) if \( E_{ij} \) acts within \( V_{\overline{0}} \) or \( V_{\overline{1}} \) (even), and \( |E_{ij}| = \overline{1} \) if it maps between \( V_{\overline{0}} \) and \( V_{\overline{1}} \) (odd).

In the general linear Lie superalgebra \( \mathfrak{g} = \mathfrak{gl}(m|n) \), the even part is given by \(
\mathfrak{g}_{\overline{0}} = \mathfrak{gl}(m) \oplus \mathfrak{gl}(n).
\)

We fix the standard Cartan subalgebra
\(
\mathfrak{h} := \bigoplus_{1 \le i \le m+n} k E_{ii}.
\)
Define linear functionals
$\varepsilon_1, \dots, \varepsilon_{n+m} \in \mathfrak{h}^*$
by requiring that
$\varepsilon_i(E_{jj}) = \delta_{ij}$ for $1 \le i,j \le n+m$.
Define \( \delta_i = \varepsilon_{m+i} \) for \( 1 \leq i \leq n \).
The non-degenerate symmetric bilinear form \( (\, , \,) \) on is defined as follows:
\(
(\varepsilon_i, \varepsilon_j) = 
(-1)^{|v_i|}\delta_{i,j}
\)

The root space \( \mathfrak{g}_\alpha \) associated with 
\( \alpha \in \mathfrak{h}^* \) is defined as
\(
\mathfrak{g}_\alpha := 
\{\, x \in \mathfrak{g} \mid [h, x] = \alpha(h)x, 
\text{ for all } h \in \mathfrak{h} \,\}.
\)
Then we have
\(
\mathfrak{g}_{\varepsilon_i - \varepsilon_j} = k E_{ij}.
\)

The set of roots \( \Delta \) is defined as
\(
\Delta := \{\, \alpha \in \mathfrak{h}^* 
\mid \mathfrak{g}_\alpha \neq 0 \,\} \setminus \{0\}.
\)

We have a root space decomposition of \( \mathfrak{g} \) with respect to \( \mathfrak{h} \):
\[
\mathfrak{g} = \mathfrak{h} 
\oplus \bigoplus_{\alpha \in \Delta} \mathfrak{g}_\alpha, 
\qquad \text{and} \quad \mathfrak{g}_0 = \mathfrak{h}.
\]

we define The sets of all even roots and odd roots as follows:
\(
\Delta_{\overline{0}} = 
\{\, \varepsilon_i - \varepsilon_j, \ \delta_i - \delta_j 
\mid i \neq j \,\},
\)
\(
\Delta_{\overline{1}}  = 
\{\, \varepsilon_i - \delta_j 
\mid 1 \le i \le m, \ 1 \le j \le n \,\}.
\)
Note that for any odd root $\alpha\in\Delta_{\bar1}$ we have $(\alpha,\alpha)=0$; i.e., every odd root is isotropic .

\begin{definition}
 A subset \( \Delta^+ \subset \Delta \) is called a \emph{positive system} if it satisfies:
\begin{enumerate}
    \item \( \Delta = \Delta^+ \cup (-\Delta^+) \) (disjoint union),
    \item For any \( \alpha, \beta \in \Delta^+ \), if \( \alpha + \beta \in \Delta \), then \( \alpha + \beta \in \Delta^+ \).
\end{enumerate}

Given a positive system \( \Delta^+ \), the associated \emph{fundamental system} is defined by
\(
\Pi := \{ \alpha \in \Delta^+ \mid \alpha \text{ cannot be written as } \alpha = \beta + \gamma \text{ with } \beta, \gamma \in \Delta^+ \}.
\)
Elements of \( \Pi \) are called \emph{simple roots}.

Let \( \Delta^+ \subset \Delta \) be a positive system. The corresponding \emph{Borel subalgebra} \( \mathfrak{b} \subset \mathfrak{g} \) is defined as
\(
\mathfrak{b} := \mathfrak{h} \oplus \bigoplus_{\alpha \in \Delta^+} \mathfrak{g}_\alpha,
\)
where \( \mathfrak{g}_\alpha \subset \mathfrak{g} \) is the root space corresponding to \( \alpha \in \Delta \).
Choose natural root vectors $e_\alpha\in\mathfrak g_\alpha$, $e_{-\alpha}\in\mathfrak g_{-\alpha}$ Then for any $\lambda\in\mathfrak h^*$,
\(
\lambda([e_\alpha,e_{-\alpha}])=(\lambda,\alpha).
\)

For a Borel subalgebra \( \mathfrak{b}\), we express the triangular decomposition of \( \mathfrak{g} \) as
\(
\mathfrak{g} = \mathfrak{n}^{\mathfrak{b}-} \oplus \mathfrak{h} \oplus \mathfrak{n}^{\mathfrak{b}+},
\)
where \( \mathfrak{b} = \mathfrak{h} \oplus \mathfrak{n}^{\mathfrak{b}+} \).

The sets of positive roots, even positive roots, and odd positive roots corresponding to \(\mathfrak b\) are denoted by
\(\Delta^{\mathfrak b,+}\),
\(\Delta_{\bar 0}^{\mathfrak b,+}\),
and
\(\Delta_{\bar 1}^{\mathfrak b,+}\),
respectively.
Let \(\Pi^{\mathfrak b}\) denote the set of all simple roots associated with the fixed Borel subalgebra \(\mathfrak b\),
and write
\(\Pi^{\mathfrak b}=\Pi^{\mathfrak b}_{\bar 0}\sqcup\Pi^{\mathfrak b}_{\bar 1}\)
for its decomposition into even and odd simple roots.
\end{definition}

It is clear that the root system of $\mathfrak{gl}(m|n)$, after forgetting the parity of roots,
identifies with that of $\mathfrak{gl}(m+n)$. In particular, it is useful to consider the Weyl
group $S_{m+n}$ of $\mathfrak{gl}(m+n)$. This phenomenon is specific to $\mathfrak{gl}(m|n)$ and
does not occur for other basic classical Lie superalgebras.

All positive root systems of $\mathfrak{gl}(m|n)$ are classified by total orders on the labeled
set $\{\varepsilon_1,\dots,\varepsilon_{m+n}\}$, equivalently by permutations
$\tau\in S_{m+n}$. Given $\tau$, set
\[
\Delta^+(\tau)
=\{\ \varepsilon_i-\varepsilon_j \mid \tau(i)<\tau(j)\ \}
\]
yielding a bijection $\tau\leftrightarrow\Delta^+(\tau)$; hence there are $(m+n)!$ positive systems.

The Weyl group $ S_m\times S_n$ of
$\mathfrak{g}_{\bar0}=\mathfrak{gl}(m)\oplus\mathfrak{gl}(n)$ acts on $\mathfrak h^*$ by
permuting the $\varepsilon_i$’s and the $\delta_j$’s separately. In particular, the even positive
root systems are all conjugate under $ S_m\times S_n$. We therefore fix the standard even Borel
$\mathfrak b^{\mathrm{st}}_{\bar0}$ (block upper triangular in $\mathfrak g_{\bar0}$), with even positive root system
\[
\Delta_{\bar0}^{\mathrm{st},+}
=\{\ \varepsilon_i-\varepsilon_j \mid 1\le i<j\le m\ \}
 \ \cup\
 \{\ \delta_p-\delta_q \mid 1\le p<q\le n\ \}.
\]

With this choice fixed, positive systems whose even part equals
$\Delta_{\bar0}^{\mathrm{st},+}$ are classified by $\varepsilon\delta$–sequences, equivalently by \((m¥n)\)-shuffle \[
\tau\in\mathrm{Sh}(m|n):=\bigl\{\,w\in S_{m+n}\ \big|\ 
w(1)<\cdots<w(m)\ \text{and}\ 
w(m+1)<\cdots<w(m+n)\,\bigr\}.
\]
 Given $\tau$, set
\[
\Delta^+(\tau)
=\{\ \varepsilon_i-\varepsilon_j \mid \tau(i)<\tau(j)\ \}
\]
yielding a bijection $\tau\leftrightarrow\Delta^+(\tau)$; hence there are $\frac{(m+n)!}{m!n!}$ positive systems with standard even one.

\begin{definition}\label{g.ydi}
A partition is a weakly decreasing finite sequence of positive integers; we allow the empty
partition (). For example, we write \(5+3+3+1\) as \((53^{2}1)\). We identify partitions with Young diagrams in French notation (rows
increase downward). A box has coordinates $(i,j)$ with $i$ the column (from the left) and $j$ the row (from
the bottom).
\end{definition}

Each  $\varepsilon\delta$–sequences with $n$ symbols $\varepsilon$ and $m$ symbols $\delta$ determines a lattice path from
$(n,0)$ to $(0,m)$ (left step for $\varepsilon$, up step for $\delta$); the region weakly
southeast of the path inside $(n^m)$ is a Young diagram fitting in the \(m\times n\) rectangle
$(n^m)$, and this gives a
bijection between  $\varepsilon\delta$–sequences and the set of Young diagrams fitting in the \(m\times n\) rectangle
$(n^m)$.

In what follows, we represent a Borel subalgebrawith sandard even Borel subalgebra by a partition.
In particular, we call ()
the \emph{distinguished Borel subalgebra}, and 
$(n^m)$ the \emph{anti-distinguished Borel subalgebra}.

%% file: 2.tex
\begin{definition}
We define the edge colored  graph
\(
\mathrm{Cay}(S_{m+n})
\)
as follows:

\begin{itemize}
  \item \textbf{Vertex set:} the set of all positive root systems. (The set of all Borel subalgebras).
  \item \textbf{Color set:} the quotient $\Delta / (\pm 1)$, i.e.\ pairs of opposite roots.
  \item \textbf{Edges:} for each color $\alpha \in \Delta / (\pm 1)$,
        connect two vertices if they are related by the simple reflection 
        corresponding to the root $\alpha$ .
\end{itemize}
\end{definition}

 This is precisely the Caylay graph of  Weyl groupoid of \(\mathfrak{gl}(m|n)\) in the sense of Heckenberger–Yamane. In this framework, the same graph construction extends to regular symmetrizable Kac–Moody Lie superalgebras and to Nichols algebras of diagonal type.

We write
\( r_{\alpha} \mathfrak{b} \) for the Borel obtained from \(\mathfrak{b}\) by the reflection at \(\alpha\).

Note that the definition of $\mathrm{Cay}(S_{m+n})$ is formulated purely in the language of $\mathfrak{gl}{(m+n})$ (via its root data and reflections), independent of the super decomposition of $\mathfrak{gl}(m|n)$. However, recalling parity, clearly there are two kinds of edges.  We call the edge associated to an even root an even reflection, and the edge associated to an odd root an odd reflection. It is therefore natural to consider the following graph.

\begin{definition}\label{RBgdef}
    The edge colored graph \( OR(\mathfrak{g}) \) is defined as follows:

    \begin{itemize}
  \item \textbf{Vertex set:} the set of positive root systems whose even part is the standard positive even root system \(\Delta_{\bar 0}^{\mathrm{st}+}\).
  \item \textbf{Color set:} the quotient \(\Delta_{\bar 1}/(\pm1)\), i.e.\ pairs of opposite odd roots.
  \item \textbf{Edges:} for each color \(\alpha\in\Delta_{\bar 1}/(\pm1)\), connect two vertices if they are related by the odd reflection corresponding to the odd root \(\alpha\).
\end{itemize}

\end{definition}

\begin{definition}
    
[Edge-colored finite Young's Lattice] \label{3.1lmn}
 Define an edge-colored graph \( L(m, n) \) as follows:
\begin{itemize}
    \item \textbf{Vertex set:}  the set of Young diagrams fitting in the \(m\times n\) rectangle
 ;
    \item \textbf{Color set:} \( C := \{1,2,...,n\} \times \{1,2,...,m\}  \);
    \item \textbf{Edges:} there is an edge of color \( (i, j) \) between vertices \( x \) and \( y \) if and only if \( y \) is obtained from \( x \) by adding or removing a box at coordinate \( (i, j) \), while all other boxes remain fixed.
\end{itemize}
\end{definition}

From the above discussion, the following observation is immediate.

\begin{proposition}
 we obtain a natural edge-colored graph isomorphism:
\[
    OR(\mathfrak{gl}(m|n)) \cong L(m, n) 
\]

The color set of \( OR(\mathfrak{gl}(m|n)) \) can naturally be identified with the collection of coordinates of boxes in an \( m \times n \) rectangle By assigning the color \( \varepsilon_i - \delta_j \) to the box at coordinates \( (j, m+1-i) \)

\end{proposition}

\begin{corollary}
    Two Borel subalgebras $\mathfrak b$ and $\mathfrak b'$ share the same even part
if and only if they are related by a sequence of odd reflections.
In particular, the graph $L(m,n)$ is isomorphic, as edge colored graph,
to any connected component obtained from $\mathrm{Cay}(S_{m+n})$
by deleting all edges corresponding to even roots.
\end{corollary}

 \begin{example}
In $\mathrm{Cay}(S_{2+2})=\mathrm{Cay}(S_4)$, label each vertex by the
subscripted $\varepsilon\delta$–sequence corresponding to its positive root
system, and draw the \emph{even reflections} (i.e., those in $S_2\times S_2$
swapping two $\varepsilon$’s or two $\delta$’s) as dashed edges.
Deleting all even-reflection (dashed) edges yields a disjoint union of
$|S_2\times S_2|=4$ connected components, each (as a colored graph) isomorphic to $L(2,2)$.
Here the sequence \(\varepsilon_1\,\delta_1\,\varepsilon_2\,\delta_2\) encodes the distinguished Borel subalgebra \(()\), whereas \(\delta_1\,\delta_2\,\varepsilon_1\,\varepsilon_2\) encodes the anti-distinguished Borel subalgebra \((2^{2})\).

\begin{minipage}{0.9\textwidth}
  \centering
  \begin{tikzpicture}[xscale=3, yscale=1.4,
    every node/.style={inner sep=0pt, font=\small}]
    % ===== Center square (δδεε everywhere, with subscripts) =====
    \node (A0) at (0,0.47) {$\delta_1\delta_2\varepsilon_2\varepsilon_1$}; % top
    \node (B0) at (0.47,0) {$\delta_2\delta_1\varepsilon_2\varepsilon_1$}; % right
    \node (C0) at (0,-0.47) {$\delta_2\delta_1\varepsilon_1\varepsilon_2$}; % bottom
    \node (D0) at (-0.47,0) {$\delta_1\delta_2\varepsilon_1\varepsilon_2$}; % left
    % middle cycle -> dashed
    \draw[dashed] (A0) -- (B0) -- (C0) -- (D0) -- (A0);

    % ===== Top square (swap ε1↔ε2) =====
    \node (A1) at (0,1.97) {$\varepsilon_2\delta_1\varepsilon_1\delta_2$};
    \node (C1) at (0,1.03) {$\delta_1\varepsilon_2\delta_2\varepsilon_1$};
    \node (B1) at (0.47,1.5) {$\varepsilon_2\delta_1\delta_2\varepsilon_1$};
    \node (D1) at (-0.47,1.5) {$\delta_1\varepsilon_2\varepsilon_1\delta_2$};
    \draw (A1) -- (B1) -- (C1) -- (D1) -- (A1);

    % ===== Bottom square (swap δ1↔δ2) =====
    \node (C2) at (0,-1.97) {$\varepsilon_1\delta_2\varepsilon_2\delta_1$};
    \node (A2) at (0,-1.03) {$\delta_2\varepsilon_1\delta_1\varepsilon_2$};
    \node (B2) at (0.47,-1.5) {$\delta_2\varepsilon_1\varepsilon_2\delta_1$};
    \node (D2) at (-0.47,-1.5) {$\varepsilon_1\delta_2\delta_1\varepsilon_2$};
    \draw (A2) -- (B2) -- (C2) -- (D2) -- (A2);

    % ===== Right square (swap ε1↔ε2 and δ1↔δ2) =====
    \node (B3) at (1.97,0) {$\varepsilon_2\delta_2\varepsilon_1\delta_1$};
    \node (D3) at (1.03,0) {$\delta_2\varepsilon_2\delta_1\varepsilon_1$};
    \node (A3) at (1.5,0.47) {$\varepsilon_2\delta_2\delta_1\varepsilon_1$};
    \node (C3) at (1.5,-0.47) {$\delta_2\varepsilon_2\varepsilon_1\delta_1$};
    \draw (A3) -- (B3) -- (C3) -- (D3) -- (A3);

    % ===== Left square (unchanged) =====
    \node (D4) at (-1.97,0) {$\varepsilon_1\delta_1\varepsilon_2\delta_2$};
    \node (B4) at (-1.03,0) {$\delta_1\varepsilon_1\delta_2\varepsilon_2$};
    \node (A4) at (-1.5,0.47) {$\delta_1\varepsilon_1\varepsilon_2\delta_2$};
    \node (C4) at (-1.5,-0.47) {$\varepsilon_1\delta_1\delta_2\varepsilon_2$};
    \draw (A4) -- (B4) -- (C4) -- (D4) -- (A4);

    % ===== Additional edges (solid) =====
    \draw (A0) -- (C1);
    \draw (B0) -- (D3);
    \draw (C0) -- (A2);
    \draw (D0) -- (B4);

    % ===== External nodes =====
    \node (NA1) at (0,2.5) {$\varepsilon_2\varepsilon_1\delta_1\delta_2$};
    \node (NB3) at (2.5,0) {$\varepsilon_2\varepsilon_1\delta_2\delta_1$};
    \node (NC2) at (0,-2.5) {$\varepsilon_1\varepsilon_2\delta_2\delta_1$};
    \node (ND4) at (-2.5,0) {$\varepsilon_1\varepsilon_2\delta_1\delta_2$};

    % Connect to external nodes (solid)
    \draw (A1) -- (NA1);
    \draw (C2) -- (NC2);
    \draw (B3) -- (NB3);
    \draw (D4) -- (ND4);

    % ===== Outer cycle -> dashed =====
    \draw[dashed] (NA1) -- (NB3) -- (NC2) -- (ND4) -- (NA1);

    % ===== More requested edges -> dashed =====
    \draw[dashed] (B1) -- (A3);
    \draw[dashed] (C3) -- (B2);
    \draw[dashed] (D2) -- (C4);
    \draw[dashed] (A4) -- (D1);
  \end{tikzpicture}
\end{minipage}

\end{example}

%% file: 3.tex
\begin{definition}

We denote by \( s\mathcal{W}\) the category of locally finite \( \mathfrak{h} \)-semisimple  \( \mathfrak{g} \)-modules (i.e. weight modules with finite dimensional weight spaces).

A weight module \( M \) admits a weight space decomposition :
\[
M = \bigoplus_{\lambda \in \mathfrak{h}^*} M_\lambda, \qquad
M_\lambda := \{ v \in M \mid h \cdot v = \lambda(h)v \ \text{for all } h \in \mathfrak{h} \}.
\]
\end{definition}

\begin{lemma}[See also \cite{brundan2014representations} Lemma 2.2,  \cite{chen2020primitive} Proposition 2.2.3]
We can choose \( p \in \operatorname{Map}(\mathfrak{h}^*, \mathbb{Z}/2\mathbb{Z}) \) such that
\[
\mathcal{W}:= \{ M \in s\mathcal{W}\mid \deg M_{\lambda} = p(\lambda) \text{ for } \lambda \in \mathfrak{h}^*, M_\lambda \neq 0 \}
\]
forms a Serre subcategory, \(  \mathcal{W}\) contains the trivial module and \( s\mathcal{W}= \mathcal{W}\oplus \Pi \mathcal{W}\).

\end{lemma}
Similarly, for \(\mathfrak g_{\bar 0}\) we introduce the category \(s\mathcal{W}_{\bar 0}\), and define the category \(\mathcal{W}_{\bar 0}\) so as to be compatible with the restriction functor \(\operatorname{Res}^{\mathfrak g}_{\mathfrak g_{\bar 0}}\).

\begin{lemma}[\cite{bell1993theory}, Th.\ 2.2; \cite{coulembier2017gorenstein}, §6.1]\label{indres}
Let \(\operatorname{Res}^{\mathfrak g}_{\mathfrak g_{\bar 0}}:\mathfrak g\text{-sMod}\to\mathfrak g_{\bar 0}\text{-sMod}\),
\(\operatorname{Ind}^{\mathfrak g}_{\mathfrak g_{\bar 0}}:\mathfrak g_{\bar 0}\text{-sMod}\to\mathfrak g\text{-sMod}\), and
\(\operatorname{Coind}^{\mathfrak g}_{\mathfrak g_{\bar 0}}:\mathfrak g_{\bar 0}\text{-sMod}\to\mathfrak g\text{-sMod}\)
denote restriction, induction, and coinduction, respectively; thus
\(\operatorname{Ind}^{\mathfrak g}_{\mathfrak g_{\bar 0}}\dashv \operatorname{Res}^{\mathfrak g}_{\mathfrak g_{\bar 0}}\dashv \operatorname{Coind}^{\mathfrak g}_{\mathfrak g_{\bar 0}}\).

Then:
(1) each of \(\operatorname{Res}^{\mathfrak g}_{\mathfrak g_{\bar0}}\), \(\operatorname{Ind}^{\mathfrak g}_{\mathfrak g_{\bar0}}\), \(\operatorname{Coind}^{\mathfrak g}_{\mathfrak g_{\bar0}}\) is exact and restricts between \(\mathcal W_{\bar0}\) and \(\mathcal W\);
(2) there is a natural isomorphism of functors \(\operatorname{Ind}^{\mathfrak g}_{\mathfrak g_{\bar0}}\cong \operatorname{Coind}^{\mathfrak g}_{\mathfrak g_{\bar0}}\);
(3) The unit \(\operatorname{Id}_{\mathcal O_{\bar0}}\!\to\!\operatorname{Res}^{\mathfrak g}_{\mathfrak g_{\bar0}}\!\circ\!\operatorname{Ind}^{\mathfrak g}_{\mathfrak g_{\bar0}}\) is a split monomorphism, and the counit \(\operatorname{Res}^{\mathfrak g}_{\mathfrak g_{\bar0}}\!\circ\!\operatorname{Ind}^{\mathfrak g}_{\mathfrak g_{\bar0}}\twoheadrightarrow \operatorname{Id}_{\mathcal W_{\bar0}}\) is a split epimorphism.
(4) \(\operatorname{Res}^{\mathfrak g}_{\mathfrak g_{\bar0}}\circ \operatorname{Ind}^{\mathfrak g}_{\mathfrak g_{\bar0}}\cong \Lambda(\mathfrak g_{\bar1})\otimes -\) on \(\mathcal W_{\bar0}\).
\end{lemma}

Note that \( \Lambda^{\dim \mathfrak{g}_{\overline{1}}}(\mathfrak{g}_{\overline{1}}) \cong \Lambda^{2mn}(\mathfrak{g}_{\overline{1}}) \cong L_{\bar0}(0)\) in our general linear setting, but in general,  we have \(\operatorname{Ind}^{\mathfrak g}_{\mathfrak g_{\bar0}}\cong \operatorname{Coind}^{\mathfrak g}_{\mathfrak g_{\bar0}}\circ(-\otimes \Lambda^{\dim\mathfrak g_{\bar1}}\mathfrak g_{\bar1})\).
 \( L_{\bar0}(0) = \Lambda^0(\mathfrak{g}_{\overline{1}}) \) and \( \Lambda^{\dim \mathfrak{g}_{\overline{1}}}(\mathfrak{g}_{\overline{1}}) \) are direct summands of \( \Lambda(\mathfrak{g}_{\overline{1}}) \). 
 
\begin{definition}
For a module \(M\) in the category \(\mathcal W\), the character \(\operatorname{ch} M\) is a formal sum that encodes the dimensions of the weight spaces of \(M\): \(\operatorname{ch} M := \sum_{\lambda} \dim M_{\lambda} e^{\lambda}\).
\end{definition}

Lie superalgebra \(\mathfrak{gl}(m|n)\) has an antiautomorphism \(\tau\) given by the formula \(\tau(E_{ij}) := -(-1)^{|i|(|i|+|j|)}\,E_{ji}\).
For a weight module \(M=\bigoplus_{\nu\in\mathfrak h^{*}} M_\nu\), we let \(M^\vee := \bigoplus_{\nu} M_\nu^{*}\) be the restricted dual of \(M\).
We may define an action of \(\mathfrak{gl}(m|n)\) on \(M^\vee\) by \((g\cdot f)(x):= -\,f(\tau(g)x)\), \(f\in M^\vee\), \(g\in \mathfrak{gl}(m|n)\), \(x\in M\).
This defines a contravariant duality functor \( (-)^\vee : \mathcal W \to \mathcal W \).

\begin{proposition}[Properties of contragredient duality]\label{prop:duality_properties}

\begin{enumerate}
  \item $(-)^\vee$ is an exact contravariant functor.
  \item There is a natural isomorphism $M \xrightarrow{\sim} M^{\vee\vee}$ for all $M\in\mathcal W$ .
  \item $\operatorname{ch}(M^\vee)=\operatorname{ch}(M)$; in particular, the weight multiplicities are preserved.
  \item Since \(\tau\) is an automorphism of \(\mathfrak g\) compatible with the inclusions of subalgebras, there is a natural isomorphism of functors \(\mathrm{Res}_{\mathfrak g_{\bar 0}}^{\mathfrak g}\cong(\,\cdot\,)^{\vee}\circ \mathrm{Res}_{\mathfrak g_{\bar 0}}^{\mathfrak g}\circ(\,\cdot\,)^{\vee}\). Consequently, \(\mathrm{Ind}_{\mathfrak g_{\bar 0}}^{\mathfrak g}\cong(\,\cdot\,)^{\vee}\circ \mathrm{Ind}_{\mathfrak g_{\bar 0}}^{\mathfrak g}\circ(\,\cdot\,)^{\vee}\).
\end{enumerate}
\end{proposition}

%% file: 4.tex
Let \( k_{\lambda}^{\mathfrak{b}} \) be the one-dimensional \( \mathfrak{b} \)-module corresponding to \( \lambda \in \mathfrak{h}^* \). The \( \mathfrak{b} \)-Verma module with highest weight \( \lambda \) is defined by
\(
M^{\mathfrak{b}}(\lambda) = U(\mathfrak{g}) \otimes_{U(\mathfrak{b})} k_{\lambda}^{\mathfrak{b}}.
\)
(Here, the parity of \( k_{\lambda}^{\mathfrak{b}} \) is chosen so that \( M^{\mathfrak{b}}(\lambda) \in \mathcal{W} \). )

Its simple top is denoted by \( L^{\mathfrak{b}}(\lambda) \). 
Similarly, for the even part \( \mathfrak{g}_{\overline{0}} \), the corresponding Verma module and simple module are denoted by \( M_{\overline{0}}(\lambda) \) and  \( L_{\overline{0}}(\lambda) \), respectively.

We define Berezinian  weight
\(
ber:=\varepsilon_1+\cdots+\varepsilon_m-(\delta_1+\cdots+\delta_n),
\)
A weight \(\lambda\) is orthogonal to all roots if and only if \(\lambda=tber\) for some \(t\in k\) (i.e. a scalar multiple of the Berezinian weight.
\(\mathrm{Ber}:= L^{()}( ber)\ \). 
For any \(t\in k\), the simple highest weight module \(L^{()}(t ber)\) is one-dimensional; indeed \(L^{()}(tber)\cong \mathrm{Ber}^{\otimes t}\). 

\begin{definition}[integral Weyl vectors \cite{brundan2014representations}]
Define the following vectors in \( \mathfrak{h}^* \)
\[
\rho_{\overline{0}} :=  \frac{1}{2} \sum_{\beta \in \Delta_{\overline{0}}^+} \beta, \quad
\rho_{\overline{1}}^{\mathfrak{b}}:=  \frac{1}{2} \sum_{\gamma \in \Delta_{\overline{1}}^{\mathfrak{b} +}} \gamma.
\]
We write \(\rho_{\overline{1}}:=\rho_{\overline{1}}^{()}\).
Here integral Weyl vector \(\rho:=-\varepsilon_2-2\varepsilon_3-\cdots-(m-1)\varepsilon_n+(m-1)\delta_{m+1}+(m-2)\delta_{m+2}+\cdots+(m-n)\delta_{m+n}\).

One checks that \(\rho=\rho_{\overline{0}} - \rho_{\overline{1}}^{()}-\frac{1}{2}(m+n-1)ber\)

\[\rho^\mathfrak{b} := \rho_{\overline{0}} - \rho_{\overline{1}}^{\mathfrak{b}}-\frac{1}{2}(m+n-1)ber\]

It is worth noting that \( \rho^{r_\alpha \mathfrak{b}} = \rho^{\mathfrak{b}} + \alpha \) for a b-simple root  \(\alpha\).
If \( \alpha \in \Pi^{\mathfrak{b}} \), then we have \( (\rho^{\mathfrak{b}}, \alpha_i^{\mathfrak{b}}) = \frac{1}{2} (\alpha, \alpha) \). In particular, if \( \alpha\) is isotropic, then we have \( (\rho^{\mathfrak{b}}, \alpha) = 0 \).

\end{definition}

\begin{lemma}\label{5.2.ch_eq}
For \( \mathfrak{b}, \mathfrak{b}' \in L(m,n) \) and \( \lambda ,\lambda'\in \mathfrak{h}^* \), the following statements hold:
\begin{enumerate}
    \item \( \operatorname{ch} M^{\mathfrak{b}}(\lambda - \rho^{\mathfrak{b}}) = \operatorname{ch} M^{\mathfrak{b}'}(\lambda - \rho^{\mathfrak{b}'}) \);
    \item \( \dim M^{\mathfrak{b}'}(\lambda - \rho^{\mathfrak{b}'})_{\lambda - \rho^{\mathfrak{b}}} = 1 \);
    \item \( \operatorname{ch} M^{\mathfrak{b}}(\lambda) = \operatorname{ch} M^{\mathfrak{b}'}(\lambda') \iff \lambda + \rho^{\mathfrak{b}} = \lambda' + \rho^{\mathfrak{b}'} \);
    \item \( \dim \operatorname{Hom}(M^{\mathfrak{b}}(\lambda - \rho^{\mathfrak{b}}), M^{\mathfrak{b}'}(\lambda - \rho^{\mathfrak{b}'})) = 1 \);
    \item If \( M^{\mathfrak{b}}(\lambda - \rho^{\mathfrak{b}}) \not\cong M^{\mathfrak{b}'}(\lambda - \rho^{\mathfrak{b}'}) \), then \( M^{\mathfrak{b}}(\lambda - \rho^{\mathfrak{b}}) \) is not isomorphic to any \( \mathfrak{b}' \)-Verma module.
    \item     Furthermore, let $\mathfrak{b}$ and $\mathfrak{b}'$ be related by an odd reflection with respect to a
$\mathfrak{b}$–simple odd root~$\alpha$.\begin{enumerate}

        \item if \( (\lambda, \alpha) \neq 0 \), then we have
\(M^{\mathfrak{b}}(\lambda - \rho^{\mathfrak{b}}) \cong M^{\mathfrak{b}'}(\lambda - \rho^{\mathfrak{b}'}) , \quad
L^{\mathfrak{b}}(\lambda - \rho^{\mathfrak{b}}) \cong L^{\mathfrak{b}'}(\lambda - \rho^{\mathfrak{b}'}).
\)
\item If \( (\lambda, \alpha) = 0 \), then we have
\(
M^{\mathfrak{b}}(\lambda - \rho^{\mathfrak{b}}) \not\cong M^{\mathfrak{b}'}(\lambda - \rho^{\mathfrak{b}'}) , \quad
L^{\mathfrak{b}}(\lambda) \cong L^{\mathfrak{b}'}(\lambda) .
\)
    \end{enumerate}
\end{enumerate}
\end{lemma}

%% file: 12.tex
\begin{definition}
Let \( \mathfrak{b} \in  L(m,n) \).  The category $\mathcal{O}^{\mathfrak b}$ is defined as the Serre subcategory of $\mathcal{W}$ generated by 
\(
\{ L^{\mathfrak{b}}(\lambda) \mid \lambda \in \mathfrak{h}^* \}.
\)
According to \cref{5.2.ch_eq}, as an Serre subcategory, it depends only on \( \mathfrak{b}_{\overline{0}} \) (however,the highest weight structure depends on \( \mathfrak{b} \)).
\end{definition}

It is well known that the BGG category $\mathcal O$ contains all $\mathfrak b$–Verma modules
for any  \( \mathfrak{b} \in  L(m,n) \).
Moreover, $\mathcal O$ is a finite–length abelian category.
Similarly, by replacing \( \mathfrak{g} \) with \( \mathfrak{g}_{\overline{0}} \), we define \( \mathcal{O}_{\overline{0}} \) as a full subcategory of \( \mathcal{W}_{\overline{0}} \) .

For a Borel subalgebra \(\mathfrak b\), we call a module of the form
\(M^{\mathfrak b}(\lambda)^\vee\), i.e. the contragredient dual of the
\(\mathfrak b\)–Verma module, a \emph{\(\mathfrak b\)–dual Verma module}.
The contragredient dual functor $(-)^\vee$ restricts to the BGG category, i.e.\ $(-)^\vee:\mathcal O^{\mathrm{op}}\to\mathcal O$ is exact and contravariant; in particular $M^{\mathfrak b}(\lambda)^\vee$ belongs to $\mathcal O$ and $L^{\mathfrak b}(\lambda)^\vee\simeq L^{\mathfrak b}(\lambda)$.

Since each of \(\operatorname{Res}^{\mathfrak g}_{\mathfrak g_{\bar0}}\), \(\operatorname{Ind}^{\mathfrak g}_{\mathfrak g_{\bar0}}\), \(\operatorname{Coind}^{\mathfrak g}_{\mathfrak g_{\bar0}}\) is exact and restricts between \(\mathcal O_{\bar0}\) and \(\mathcal O\), for each choice of Borel subalgebra \(\mathfrak b\), the category \(\mathcal O\) carries a highest weight structure with standard objects the \(\mathfrak b\)–Verma modules \(M^{\mathfrak b}(\lambda)\) (and costandard objects \((M^{\mathfrak b}(\lambda))^\vee\)); the weight poset is \((\mathfrak h^*,\le_{\mathfrak b})\) where \(\mu\le_{\mathfrak b}\lambda\) iff \(\lambda-\mu\) is a sum of roots in \(\Delta^{\mathfrak b+}\).
For a weight $\lambda$, 
we write $P(\lambda)$ and $I(\lambda)$ for its projective cover and injective hull in $\mathcal O$, respectively.
Similarly, we write
$P_{\bar0}(\lambda)$, $I_{\bar0}(\lambda)$ for  projective cover and injective hull of  $L_{\bar0}(\lambda)$ in in $\mathcal O_{\bar0}$ , respectively .

For $\mathcal O$ one has BGG reciprocity
\(
\big[P^{\mathfrak b}(\lambda):M^{\mathfrak b}(\mu)\big]
\;=\;
\big[M^{\mathfrak b}(\mu):L^{\mathfrak b}(\lambda)\big].
\)  

For fixed \(\mathfrak b\in B(\mathfrak g)\), an object \(M\in\mathcal O\) has a \(\mathfrak b\)-Verma flag if it admits a finite filtration \(0=F_0\subset F_1\subset\cdots\subset F_r=M\) with subquotients \(F_i/F_{i-1}\cong M^{\mathfrak b}(\lambda_i)\). Let \(\mathcal F\Delta^{\mathfrak b}\subset\mathcal O\) denote the full subcategory of objects with a \(\mathfrak b\)-Verma flag. Similarly, we define the category of modules with a \(\mathfrak b\)-dual Verma flag, \(\mathcal{F}\!\nabla^{\mathfrak b}\subset\mathcal O\), and the category of modules with an even Verma flag, \(\mathcal{F}\!\Delta_{\bar 0}\subset\mathcal O_{\bar 0}\).

\begin{proposition}\label{O.indproj}
Let \( \lambda \in \mathfrak{h}^* \).  
Fix a Borel subalgebra \( \mathfrak{\bar{b}} \).

1.  The category \( \mathcal{F}{\Delta^{\mathfrak{\bar{b}}}} \) is closed under direct summands.
2.\( \operatorname{Ind}_{\mathfrak{g}_{\overline{0}}}^{\mathfrak{g}} M_{\overline{0}}(\lambda) ,
\operatorname{Ind}_{\mathfrak{g}_{\overline{0}}}^{\mathfrak{g}} P_{\overline{0}}(\lambda) , P^{\mathfrak{\bar{b}}}(\lambda) \in \bigcap_{\mathfrak{b} \in \mathfrak{B}} \mathcal{F}{\Delta^{\mathfrak{b}}}; \)
3.  \(\operatorname{Res}^{\mathfrak g}_{\mathfrak g_{\bar 0}} M^{\mathfrak {\bar{b}}}(\lambda)\in \mathcal{F}\!\Delta_{\bar 0}.\)

\end{proposition}

%% file: 14.tex
Recall that we represent a Borel subalgebrawith sandard even Borel subalgebra by a partition.
In particular, we denote ()
the \emph{distinguished Borel subalgebra}, and 
$(n^m)$ the \emph{anti-distinguished Borel subalgebra}.

Define
\(
\mathfrak g_{1}:=\mathfrak g_{\bar1}\cap()\qquad
\mathfrak g_{-1}:=\mathfrak g_{\bar1}\cap (n^m).
\)

The general linear Lie superalgebra \( \mathfrak{gl}(m|n) \) admits a natural \( \mathbb{Z} \)-grading given by
\(
\mathfrak{gl}(m|n) = \mathfrak{g}_{-1} \oplus \mathfrak{g}_{\bar0} \oplus \mathfrak{g}_{1}.
\)
This grading satisfies the relations
\(
[\mathfrak{g}_1, \mathfrak{g}_1] = [\mathfrak{g}_{-1}, \mathfrak{g}_{-1}]=0.
\)
We define   \(
\mathfrak{g}_{\geq 0} := \mathfrak{g}_{\bar0} \oplus \mathfrak{g}_1,
\quad
\mathfrak{g}_{\leq 0} := \mathfrak{g}_{-1} \oplus \mathfrak{g}_{\bar0}.
\)

Define the inflation functor \(\mathrm{Infl}_{\mathfrak g_{\bar 0}}^{\mathfrak g_{\ge 0}}\colon 
\mathfrak g_{\bar 0}\text{-sMod}\to \mathfrak g_{\ge 0}\text{-sMod}\) by letting 
\(\mathfrak g_{1}\) act trivially. 
Define the functor \((\,)^{\mathfrak g_{1}}\colon 
\mathfrak g_{\ge 0}\text{-sMod}\to \mathfrak g_{\bar 0}\text{-sMod}\) 
by \(M\mapsto M^{\mathfrak g_{1}}:=\{\,m\in M\mid \mathfrak g_{1}m=0\,\}\) . 
Then these functors are mutually adjoint each other  and exact.

The Kac functor is defined as 
\(K_{\ge 0}:=\mathrm{Ind}_{\mathfrak g_{\ge 0}}^{\mathfrak g}\circ 
\mathrm{Infl}_{\mathfrak g_{\bar 0}}^{\mathfrak g_{\ge 0}}:
\mathfrak g_{\bar 0}\text{-sMod}\to \mathfrak g\text{-sMod}\), which is exact. 
Then \(K_{\ge 0}\) is left adjoint to the exact functor 
\(\mathrm{Res}_{\ge 0}:=(\,)^{\mathfrak g_{1}}
\circ \mathrm{Res}_{\mathfrak g_{\ge 0}}^{\mathfrak g}:\mathfrak g\text{-sMod}\to \mathfrak g_{\bar 0}\text{-sMod}\).
If we denote by \(\mathrm{Coind}_{\mathfrak g_{\ge 0}}^{\mathfrak g}\) \cite{chen2021simple}
the right adjoint of \(\mathrm{Res}_{\mathfrak g_{\ge 0}}^{\mathfrak g}\), and define  \(\mathrm{Coind}_{\ge 0}:=\mathrm{Coind}_{\mathfrak g_{\ge 0}}^{\mathfrak g}\circ 
\mathrm{Infl}_{\mathfrak g_{\bar 0}}^{\mathfrak g_{\ge 0}}:
\mathfrak g_{\bar 0}\text{-sMod}\to \mathfrak g\text{-sMod}\) then there is an isomorphism of functors
\(
\mathrm{Coind}_{\ge 0}
\circ(-\otimes L_{\bar 0}(-2\rho_{\bar 1}))\cong
K_{\ge 0}
\), see \cite{chen2021simple}.  Note that \(L_{\bar 0}(-2\rho_{\bar 1}) \cong \bigwedge^{\dim \mathfrak g_{-1}} \mathfrak g_{-1}\) as \(\mathfrak g_{\bar 0}\)-modules.

Similarly, we define the dual Kac functor 
\(K_{\le 0}:=\mathrm{Ind}_{\mathfrak g_{\le 0}}^{\mathfrak g}\circ 
\mathrm{Infl}_{\mathfrak g_{\bar 0}}^{\mathfrak g_{\le 0}}:
\mathfrak g_{\bar 0}\text{-sMod}\to \mathfrak g\text{-sMod}\),
\(\mathrm{Res}_{\le 0}:=\mathrm{Res}_{\mathfrak g_{\le 0}}^{\mathfrak g}\circ(\,\cdot\,)^{\mathfrak g_{-1}}\) and
 \(\mathrm{Coind}_{\le 0}\).

\begin{lemma}\cite{germoni1998indecomposable}
    \(
K_{\ge 0}\cong(\,\cdot\,)^{\vee}\circ
\mathrm{Coind}_{\le 0}\circ(\,\cdot\,)^{\vee},
\qquad
K_{\le 0}\cong(\,\cdot\,)^{\vee}\circ
\mathrm{Coind}_{\ge 0}\circ(\,\cdot\,)^{\vee}.
\)
\(
K_{\ge 0}\) and \(
K_{\le 0}\) restricts between \(\mathcal O_{\bar0}\) and \(\mathcal O\).
\end{lemma}

\begin{proposition}\label{kacverma}
For every weight \(\lambda\) one has \(K_{\ge 0}\big(M_{\bar 0}(\lambda)\big)\cong M^{()}(\lambda)\) and \(K_{\le 0}\big(M_{\bar 0}(\lambda)\big)\cong M^{(n^{m})}(\lambda)\). Using \((\,\cdot\,)^{\vee}\), \(\mathrm{Coind}_{\le 0}\big(M_{\bar 0}(\lambda)^{\vee}\big)\cong K_{\le 0}\big(M_{\bar 0}(\lambda-2\rho_{\bar 1})^{\vee}\big)\cong M^{()}(\lambda)^{\vee}\), and \(\mathrm{Coind}_{\ge 0}\big(M_{\bar 0}(\lambda)^{\vee}\big)\cong K_{\ge 0}\big(M_{\bar 0}(\lambda+2\rho_{\bar 1})^{\vee}\big)\cong M^{(n^{m})}(\lambda)^{\vee}\).
We also have  $K_{\ge 0}(L_{\bar 0}(\lambda))^{\vee}\cong K_{\le 0}(L_{\bar 0}(\lambda-2\rho_{\bar 1}))$, and $\operatorname{soc}K_{\ge 0}(L_{\bar 0}(\lambda))\cong L^{(n^m)}(\lambda-2\rho_{\bar 1})$.
\end{proposition}

\begin{theorem}\cite{chen2021translated}
For every \(M\in\mathcal O_{\bar 0}\) one has \(\ell(\operatorname{soc}M)=\ell(\operatorname{soc}K_{\ge 0}M)=\ell(\operatorname{soc}K_{\le 0}M)\) and \(\ell(\operatorname{top}M)=\ell(\operatorname{top}K_{\ge 0}M)=\ell(\operatorname{top}K_{\le 0}M)\), where \(\ell(-)\) denotes the length.
\end{theorem}

\begin{proposition}(see \cite{mazorchuk2014parabolic,chen2021tilting})\label{antidom}
For a weight \(\lambda\), the following are equivalent:
\begin{enumerate}
\item \(M_{\bar 0}(\lambda)\cong L_{\bar 0}(\lambda)\cong M_{\bar 0}(\lambda)^{\vee}\).
\item \(P_{\bar 0}(\lambda)\cong I_{\bar 0}(\lambda)\).
\item \(P^{()}(\lambda)\cong I^{()}(\lambda)\).
\item \(L^{()}(\lambda)\) is the socle of some \(M\) s.t. \(\operatorname{Res}^{\mathfrak g}_{\mathfrak g_{\bar 0}} M\in \mathcal{F}\!\Delta_{\bar 0}.\).
\item \(M^{()}(\lambda)\cong M^{(n^{m})}(\lambda-2\rho_{\bar 1})^{\vee}\).
\end{enumerate}
\end{proposition}

\begin{proof}
(1)\(\Leftrightarrow\)(2) are classical.
(2)\(\Leftrightarrow\)(3) follows from \cref{indres}.
(3)\(\Rightarrow\)(4) is trivial
(4)\(\Rightarrow\)(3): see \cite{mazorchuk2014parabolic,irving1985projective}.
(1)\(\Leftrightarrow\)(5): \(M^{()}(\lambda)^{\vee}\cong K_{\le 0}\big(M_{\bar 0}(\lambda-2\rho_{\bar 1})^{\vee}\big)\cong K_{\le 0}\big(M_{\bar 0}(\lambda-2\rho_{\bar 1})\big)\cong M^{(n^{m})}(\lambda-2\rho_{\bar 1}).\)
\end{proof}

We call such a weight \(\lambda\) (or \(L^{()}(\lambda)\)) antidominant.

%% file: 25.tex
\begin{definition}[Edge contraction in simple graphs]
A (finite) simple undirected graph is a pair \(G=(V,E)\) with
\(E\subseteq \binom{V}{2}=\{\{u,v\}\mid u,v\in V,\ u\neq v\}\)
(i.e.\ no loops and no parallel edges).

Let \(S\subseteq E\) be a set of edges. Define an equivalence relation \(\sim_S\) on \(V\) by
\[
u\sim_S v \iff \text{$u$ and $v$ lie in the same connected component of the spanning subgraph }(V,S).
\]
The \emph{contraction of \(S\)} is the simple graph
\[
G/S=(V',E')\quad\text{with}\quad V':=V/{\sim_S},
\]
where \([v]\) denotes the \(\sim_S\)-class of \(v\) and
\[
E'\ :=\ \bigl\{\ \{\,[a],[b]\,\}\ \bigm|\ \{a,b\}\in E\setminus S,\ [a]\neq[b]\ \bigr\}.
\]
Thus any loop that would arise from an edge whose endpoints merge (\([a]=[b]\)) is discarded by definition,
and parallel edges collapse automatically since \(E'\subseteq \binom{V'}{2}\).
\end{definition}

\begin{remark}
The edge–contraction construction does not, in general, produce a new simple edge–colored graph in the abstract setting of edge-colored graphs. Remarkably, for the finite Young lattices \(L(m,n)\) (and, more generally, for the odd reflection graphs arising from regular symmetrizable Kac–Moody Lie superalgebras and from Nichols algebras of diagonal type), a special property inherited from the Weyl groupoid—what we call the “rainbow boomerang graph \cite{Hirota_PathSubgroupoids}” —ensures that the construction can be carried out in a way compatible with the coloring.
\end{remark}

\begin{definition}[Color quotient via edge contraction]
Let $L(m,n)=(V,E)$ be as defined before, with color set $\mathcal C$.
For a subset $A\subseteq\mathcal C$, put
\[
L(m,n)/A\ :=\ L(m,n)/\{e\in E\mid c(e)\in A\},
\]
and equip it with the color map $\tilde c$ obtained by restricting $c$ to the
surviving edges (i.e.\ edges from $E\setminus S$ whose endpoints do not merge).
\end{definition}

\begin{lemma}[Well-definedness for $L(m,n)$]
For $L(m,n)$ the map $\tilde c$ is well defined: whenever two edges of
$E\setminus S$ project to the same edge of $L(m,n)/A$, they have the
same color in $\mathcal C\setminus A$.

\end{lemma}

% Definition of \( OR(\mathfrak{g}) \)
\begin{definition}
    For \( \lambda \in \mathfrak{h}^* \)  we define an edge colored  connected simple graph \( L(m,n)_\lambda:= L(m,n)/\{ \alpha \in \Delta_{\bar1}/(\pm 1)  \mid (\lambda, \alpha) \neq 0 \}\) . 
See \cref{l32}.  Note that  \(L(m,n)_0:= L(m,n)\).
    We call \(L(m,n)_{\lambda + \rho^{\mathfrak{b}}}\)   the odd reflection graph of  \(
M^{\mathfrak{b}}(\lambda)\).
\end{definition}

% Corollary 5.2
\begin{corollary}\label{5.2RB}
The vertex set of \( L(m,n)_{\lambda}\) can be identified with the isomorphism classes of \({ \{ M^{\mathfrak{b}}(\lambda - \rho^{\mathfrak{b}}) \} }_{\mathfrak{b} }\).
Precisely,  we have:
\[
M^{\mathfrak{b}}(\lambda  - \rho^{\mathfrak{b}}) \cong M^{\mathfrak{b}'}(\lambda   -\rho^{\mathfrak{b}'}) \iff  [\mathfrak{b}]= [\mathfrak{b'}] \in   L(m,n)_{\lambda}\]
\end{corollary}
\begin{proof}
By \cref{5.2.ch_eq}, whether  two adjacent Verma modules are isomorphic depends only on the label \(\alpha\), i.e. on the color of the edge. This is exactly the claim.
\end{proof}

\begin{definition}
Let \( \lambda \in \mathfrak{h}^* \) and \( \mathfrak{b} \in L(m,n) \). 
 We define the \( \mathfrak{b} \)-\emph{atypicality} of \( \lambda \) as
\[
\mathrm{aty}^{\mathfrak{b}}(\lambda) := \max \left\{ t \;\middle|\;
\begin{array}{l}
\text{there exist mutually orthogonal distinct roots } \alpha_1, \dots, \alpha_t \in \Delta_{\bar1} /(\pm 1)\\
\text{such that } (\lambda + \rho^{\mathfrak{b}}, \alpha_i) = 0 \quad \text{for all } i = 1, \dots, t
\end{array}
\right\}.
\]
If \( \mathrm{aty}^{\mathfrak{b}}(\lambda) = 0 \), then \( \lambda \) is called \( \mathfrak{b} \)-\emph{typical}; otherwise, it is called \( \mathfrak{b} \)-\emph{atypical}.

We def aty$L^{\mathfrak b}(\lambda) := \mathrm{aty}^{\mathfrak{b}}(\lambda) $ 
If \( \mathrm{aty}^{\mathfrak{b}}(\lambda) = 0 \), then \( L^{\mathfrak b}(\lambda) \) is called \emph{typical}; otherwise, it is called \emph{atypical}.
This definition is independent of the choice of \( \mathfrak{b} \).

We define the \emph{defect} of \( \mathfrak{g} \) as
\(
\mathrm{def}(\mathfrak{g}) := 
\max \left\{ \operatorname{aty}(L)\ \middle|\ L \text{ is a simple highest weight module} \right\}
.
\)

It is easy to check that
\(
\mathrm{def}(\mathfrak{gl}(m|n)) = \min(m,n).
\)
\end{definition}

\begin{proposition}\label{RBtriv}
Fix a reference Borel subalgebra \( \bar{\mathfrak b} \).
For \( \lambda\in\mathfrak h^* \) the following are equivalent:
\begin{enumerate}
  \item \( L^{\bar{\mathfrak b}}(\lambda - \rho^{\mathfrak{b}}) \) is typical;
  \item \( M^{\bar{\mathfrak b}}(\lambda - \rho^{\mathfrak{b}}) \in \bigcap_{\mathfrak b\in L(m,n)} \mathcal{F}\!\Delta^{\mathfrak b} \);
  \item the vertex set of \( L(m,n)_{\lambda } \) consists of a single point;
  \item \( L^{\bar{\mathfrak b}}(\lambda- \rho^{\mathfrak{b}}) \cong L^{\mathfrak b}(\lambda -\rho^{\mathfrak b}) \) for all \( \mathfrak b\).
\end{enumerate}
\end{proposition}

\begin{corollary}
If \(M^{\mathfrak b}(\lambda)\cong L^{\mathfrak b}(\lambda)\) or \(M^{\mathfrak b}(\lambda)\cong P^{\mathfrak b}(\lambda)\), then \(L^{\mathfrak b}(\lambda)\) is typical.
\end{corollary}

%% file: 26.tex
\begin{definition}
In a simple undirected graph \(G=(V,E)\), a (finite) walk of length \(k\) is a sequence \(v_0,\dots,v_k\) with \(\{v_{i-1},v_i\}\in E\) for all \(i=1,\dots,k\).
 A walk is called \emph{shortest} if there is no shorter walk between the same pair of vertices.  
 A walk is called \emph{rainbow} \cite{li2013rainbow} if all the edge colors in its sequence are distinct. 
\end{definition}

Fix a weight $\lambda$. In $L(m,n)_\lambda$, we naturally identify a length–one walk
(i.e.\ an edge ) with the unique
(up to scalars) nonzero homomorphism between the corresponding two Verma modules adjacent via an odd reflection. Consequently, any walk can be identified with the composition of such homomorphisms.
For example, a rainbow walk gives a nonzero homomorphism by the PBW Theorem. The following theorem provides a complete description of  homomorphisms between Verma modules  sharing same characters .
% Theorem
\begin{theorem}\cite{hirota2025odd}\label{5.3main}
Let \(\lambda\in\mathfrak h^{*}\) and \(w\) a walk in \(L(m,n)_{\lambda}\). Then, \[w\neq0\iff w\text{ is rainbow}\iff w\text{ is shortest}.\]

\end{theorem}

\begin{corollary}\label{oddvermacor}
Let $\mathfrak b_1,\mathfrak b_2,\mathfrak b_3\in B(\mathfrak g)$ and $\lambda\in\mathfrak h^*$.
For $i,j\in\{1,2,3\}$, let
\[
\psi^{\mathfrak b_i,\mathfrak b_j}_\lambda\in
\operatorname{Hom}\big(M^{\mathfrak b_i}(\lambda-\rho^{\mathfrak b_i}),\,M^{\mathfrak b_j}(\lambda-\rho^{\mathfrak b_j})\big)
\]
denote the nonzero (defined up to scalars) morphism .
Then the following are equivalent:

 \textbf{(1)} In $L(m,n)_\lambda$, there is a shortest walk from $[\mathfrak b_1]$ to $[\mathfrak b_3]$ passing through $[\mathfrak b_2]$.

\noindent\begin{tabular}{@{}p{0.47\linewidth}@{\quad}p{0.47\linewidth}@{}}
\textbf{(2)} $\displaystyle \psi^{\mathfrak b_2,\mathfrak b_1}_\lambda\circ\psi^{\mathfrak b_3,\mathfrak b_2}_\lambda
          =\psi^{\mathfrak b_3,\mathfrak b_1}_\lambda \neq 0$
&
\textbf{(7)} $\displaystyle \psi^{\mathfrak b_2,\mathfrak b_3}_\lambda\circ\psi^{\mathfrak b_1,\mathfrak b_2}_\lambda
          =\psi^{\mathfrak b_1,\mathfrak b_3}_\lambda \neq 0$
\\[0.6em]
\textbf{(3)} $\operatorname{Im}\psi^{\mathfrak b_3,\mathfrak b_2}_\lambda \not\subset
         \ker\psi^{\mathfrak b_2,\mathfrak b_1}_\lambda$
&
\textbf{(8)} $\operatorname{Im}\psi^{\mathfrak b_1,\mathfrak b_2}_\lambda \not\subset
         \ker\psi^{\mathfrak b_2,\mathfrak b_3}_\lambda$
\\[0.4em]
\textbf{(4)} $\operatorname{Im}\psi^{\mathfrak b_3,\mathfrak b_1}_\lambda \subset
         \operatorname{Im}\psi^{\mathfrak b_2,\mathfrak b_1}_\lambda$
&
\textbf{(9)} $\operatorname{Im}\psi^{\mathfrak b_1,\mathfrak b_3}_\lambda \subset
         \operatorname{Im}\psi^{\mathfrak b_2,\mathfrak b_3}_\lambda$
\\[0.4em]
\textbf{(5)} $\operatorname{Im}\psi^{\mathfrak b_2,\mathfrak b_1}_\lambda$ has
        $L^{\mathfrak b_3}(\lambda-\rho^{\mathfrak b_3})$ as a composition factor
&
\textbf{(10)} $\operatorname{Im}\psi^{\mathfrak b_2,\mathfrak b_3}_\lambda$ has
        $L^{\mathfrak b_1}(\lambda-\rho^{\mathfrak b_1})$ as a composition factor
\\[0.4em]
\textbf{(6)} $\ker\psi^{\mathfrak b_3,\mathfrak b_2}_\lambda \subset
         \ker\psi^{\mathfrak b_3,\mathfrak b_1}_\lambda$
&
\textbf{(11)} $\ker\psi^{\mathfrak b_1,\mathfrak b_2}_\lambda \subset
         \ker\psi^{\mathfrak b_1,\mathfrak b_3}_\lambda$
\end{tabular}
\end{corollary}

%% file: 81.tex
This subsection’s results hold, in principle, for all basic classical Lie superalgebras.

Let \(G\) be a simple graph.
For \(A,B\in G\), we say that \(B\) is \(A\)-geodesically maximal 
 if there is no vertex \(C\in G\) with \(C\neq B\) such that  there is a shortest walk \(C\) to \(A\) passes through \(B\).

\begin{lemma}\label{lmmD}
Let 
\(f: M^{\mathfrak b_2}(\lambda-\rho^{\mathfrak b_2}) \to M^{\mathfrak b_1}(\lambda-\rho^{\mathfrak b_1})\)
be a  homomorphism such that \(\mathrm{Im}\,f\) is simple. Then:
\begin{enumerate}
\item \([\mathfrak b_2]\) is \([\mathfrak b_1]\)-geodesically maximal in \(L(m,n)_{\lambda}\).
\item \(L^{\mathfrak b_2}(\lambda-\rho^{\mathfrak b_2})\) is antidominant.
\item If \(\mathrm{Im}\,f\cong \operatorname{soc} M^{\mathfrak b_1}(\lambda-\rho^{\mathfrak b_1})\), then \([\mathfrak b_2]\) is the unique \([\mathfrak b_1]\)-geodesically maximal vertex in \(L(m,n)_{\lambda}\).
\end{enumerate}
\end{lemma}

\begin{proof}
1. Assume the contrary. Then there exists \(\mathfrak b_3\) such that \([\mathfrak b_3] \neq [\mathfrak b_2] \)  and some shortest walk from \([\mathfrak b_3]\) to \([\mathfrak b_1]\) passes through \([\mathfrak b_2]\).
For each edge on this path there is a nonzero homomorphism between the corresponding Verma modules; composing them yields a nonzero map
\(g: M^{\mathfrak b_3}(\lambda-\rho^{\mathfrak b_3}) \to M^{\mathfrak b_1}(\lambda-\rho^{\mathfrak b_1})\).
By \cref{oddvermacor} one has \(\mathrm{Im}\,g \subsetneq \mathrm{Im}\,f\), so \(\mathrm{Im}\,f\) contains a proper nonzero submodule, contradicting simplicity.
2. follows from \cref{antidom}.
3. follows from \cref{oddvermacor}.
\end{proof}

\begin{corollary}\label{lemmE}
If \(M^{\mathfrak b_1}(\lambda)\cong \big(M^{\mathfrak b_2}(\mu)\big)^{\vee}\), then \([\mathfrak b_1]\) and \([\mathfrak b_2]\) are unique geodesically maximal each other in \(L(m,n)_{\lambda+ \rho^{\mathfrak{b_1}}}\).
\end{corollary}

\begin{proof}
The composition
\[
M^{\mathfrak b_2}(\mu)\twoheadrightarrow L^{\mathfrak b_2}(\mu)
\cong \operatorname{soc}M^{\mathfrak b_1}(\lambda)\hookrightarrow M^{\mathfrak b_1}(\lambda)
\]
is a nonzero homomorphism with simple  image .
Note that \(chM^{\mathfrak b_1}(\lambda)= chM^{\mathfrak b_2}(\mu)\).
By \cref{lmmD} (the “unique” clause), \([\mathfrak b_2]\) is the unique farthest vertex from \([\mathfrak b_1]\) in \(L(m,n)_{\lambda+ \rho^{\mathfrak{b_1}}}\).
Swapping \(\mathfrak b_1\) and \(\mathfrak b_2\) yields the converse.
\end{proof}

%% file: 83.tex
% (optional)
% \usepackage{amsthm}
% \newtheorem{definition}{Definition}

\begin{definition}
An edge in a graph \(G\) is called a \emph{bridge}  if removing it increases the number of connected components:
\end{definition}

The following is trivial.

\begin{lemma}\label{lemmA}
Suppose an edge \(e\) in a finite simple graph  \(G\) is a bridge and removing \(e\) yields two components \(C_{1}\) and \(C_{2}\) ..
Fix any \(c_{1}\in C_{1}\). Then there exists \(c_{2}\in C_{2}\) such that \(c_{2}\) is \(c_{1}\)-geodesically maximal.
\end{lemma}

\begin{lemma}\label{lmmB}
For any weight \(\lambda\), in \(L(m,n)_{\lambda}\) the vertex \([(n^{m})]\) is the unique  \([()]\)-geodesically maximal.
\end{lemma}
\begin{proof}
Let \(\mathcal C_{\lambda}:=\{\alpha\in\Delta_{\bar1}^{()}\mid (\lambda,\alpha)=0\}\) be the color set of \(L(m,n)_{\lambda}\).
There exists a rainbow walk from \([()]\) to \([(n^{m})]\) using each color in \(\mathcal C_{\lambda}\) exactly once; its length is \(|\mathcal C_{\lambda}|\).
Clearly this is the longest rainbow walk in \(L(m,n)_{\lambda}\).
In \(L(m,n)_{\lambda}\), the endpoint of a rainbow walk with a fixed starting vertex is determined solely by the \emph{set} of colors used. 

\end{proof}

A weight \(\lambda\) (or \(L^{()}(\lambda)\)) is \emph{integral} if \(\lambda=\sum_{i=1}^m a_i\,\varepsilon_i+\sum_{j=1}^n b_j\,\delta_j\) with  \(a_i,b_j\in\mathbb Z\)  for all \(i,j\).  

It is convenient to encode an integral weight \(\lambda\) by the \(m\mid n\)-tuple
\((\lambda_1,\dots,\lambda_m\mid \lambda_{m+1},\dots,\lambda_{m+n})\) of integers defined by
\(\lambda_i:=(\lambda+\rho,\varepsilon_i)\).
Note that our \(\rho^{\mathfrak b}\) is integral.
Write the set of integral weights as \(\Lambda\).

For an integral weight \(\lambda=(\lambda_1,\dots,\lambda_m\mid \lambda_{m+1},\dots,\lambda_{m+n})\),
\(\lambda\) (or \(L^{()}(\lambda)\)) is \emph{antidominant} iff
\(\lambda_1\le\lambda_2\le\cdots\le\lambda_m\) and
\(\lambda_{m+1}\ge\lambda_{m+2}\ge\cdots\ge\lambda_{m+n}\).

We call \(\lambda=(\lambda_1,\dots,\lambda_m\mid \lambda_{m+1},\dots,\lambda_{m+n})\) (or \(L^{()}(\lambda)\))  \emph{regular} iff
\(\lambda_i\ne\lambda_j\) for all \(1\le i<j\le m\) and  \(m+1\le i <j\le  m+n\).

\begin{definition}[Borel relabelling of a highest weight]
Let \(L=L^{()}(\lambda)\) be a simple module. For each Borel subalgebra \(\mathfrak b\), define \(\lambda_{\mathfrak b}\in\mathfrak h^{*}\) by the condition
\[
L^{\mathfrak b}(\lambda_{\mathfrak b}-\rho^{\mathfrak b})\cong L^{()}(\lambda).
\]
In particular, \(\lambda_{()}=\lambda+\rho^{()}\).
\end{definition}

\begin{lemma}\label{lemmC}
Let \(\lambda \in \Lambda\) is (regular) antidominant, then \(\lambda_{\mathfrak b}-\rho^{()}\) is also (regular) antidominant.
\end{lemma}
\begin{proof}
Fix a rainbow walk
\(
[()]\xrightarrow{\,c_1\,}[\mathfrak b_1]\xrightarrow{\,c_2\,}\cdots
\xrightarrow{\,c_k\,}[\mathfrak b],
\)
where each color \(c_i\) is identified with an odd root
\(\varepsilon_{p_i}-\delta_{q_i}\in\Delta_{\bar1}^{()+}\).
Start from \(\mu^{(0)}:=\lambda=(\mu_1,\dots,\mu_m\mid \mu_{m+1},\dots,\mu_{m+n})\), and for
\(i=1,\dots,k\) perform the operation \(O_i\) below to obtain \(\mu^{(i)}\) from \(\mu^{(i-1)}\):

\smallskip
\emph{Operation \(O_i\):} 
If \(\mu^{(i-1)}_{p_i}=\mu^{(i-1)}_{m+q_i}\), replace the pair by
\[
\big(\mu^{(i-1)}_{p},\,\mu^{(i-1)}_{m+q}\big)\ \mapsto\
\big(\mu^{(i-1)}_{p}+1,\,\mu^{(i-1)}_{m+q}+1\big),
\]
and leave all other entries unchanged; otherwise do nothing.
\smallskip

Then \(\mu^{(k)}=\lambda_{\mathfrak b}-\rho^{()}\).
This operation clearly preserves (regular) antidominance at each step.
\end{proof}

\begin{theorem}\label{lemmF}
Let  \(\mathfrak b \in L(m,n) \), \(\lambda \in \Lambda\) and suppose \(L^{\mathfrak b}(\lambda-\rho^{\mathfrak b})\) is antidominant.
Then the pair \(\{[()],[(n^{m})]\}\) in \(L(m,n)_{\lambda}\) is the unique pair of vertices that are unique geodesically maximal each other.
\end{theorem}

\begin{proof}
By \cref{lemmC}  \(L^{()}(\lambda-\rho^{()})\) is also antidominant.
By the definition of antidominance, let \(i\) be maximal with \(1\le i\le m\) such that there exists \(j\) with \(m+1\le j\le m+n\) and \(\lambda_i=\lambda_j\); choose such \(j\) minimal.
Then for any \(k\) with \(m+1\le k<j\) there is no \(l\) with \(1\le l\le m\) and \(\lambda_k=\lambda_l\).
This combinatorial condition implies that \(L(m,n)_{\lambda}\) has a bridge whose deletion isolates the singleton component \(\{[()]\}\).

By \cref{lemmA} , for any vertex \(c\ne[()]\), the vertex \([()]\) is of \(c\)-geodesically maximal.
By \cref{lmmB} ,  \([()]\) and \([(n^{m})]\) are mutually unique geodesically maximal, and such a pair is also unique.
\end{proof}

Combining \cref{lmmD}, \cref{lemmE}, and \cref{lemmF} yields a proof of \cref{b1b2}.

%% file: 92.tex
Let an integral weight \(\lambda = (\lambda_1,\dots,\lambda_m\mid \lambda_{m+1},\dots,\lambda_{m+n})\).

We say that \(\lambda\)  (or \(L^{()}(\lambda)\)) is \emph{dominant} iff
\(\lambda_{1}\ge\cdots\ge\lambda_{m}\)  and
\(\lambda_{m+1}\le\cdots\le\lambda_{m+n}\) .

To each regular dominant weight \( \lambda \in \Lambda\) (or \(L^{()}(\lambda)\)) we associate, following \cite{stroppel2012highest},  its weight diagram as following.

\(
I_{\times}(\lambda): = \{\lambda_1, \lambda_2 , \ldots, \lambda_m \}
\) and 
\(
I_{\circ}(\lambda) := \{\lambda_{m+1}, \lambda_{m+2}, \ldots,  \lambda_{m+n}\}.
\)
The integers in \( I_{\times}(\lambda) \cap I_{\circ}(\lambda) \) are labelled by \( \vee \), the remaining ones in \( I_{\times}(\lambda) \) respectively \( I_{\circ}(\lambda) \) are labelled by \( \times \) respectively \( \circ \). All other integers are labelled by a \( \wedge \).
This labelling of the number line \( \mathbb{Z} \) uniquely characterizes the weight \( \lambda \).

\begin{proposition}
Let \(L\) be a regular dominant simple module. The following are equivalent.
\begin{enumerate}
\item The weight diagram of \(L\) contains at least one \(\wedge\) between any two \(\vee\)'s.
\item For every Borel \(\mathfrak b\), the graph \(L(m,n)_{\lambda_{\mathfrak b}}\) is  a line segment.
\end{enumerate}
\end{proposition}
\begin{proof}
Exactly the same as the proof of \ref{lemmC}, mutatis mutandis.
\end{proof}

\begin{definition}
For  \(\lambda\) , we say that \(\lambda\) (or \(L^{()}(\lambda)\)) is \emph{totally disconnected} if for every Borel \(\mathfrak b\), the graph \(L(m,n)_{\lambda_{\mathfrak b}}\) is  a line segment.
\end{definition}

\begin{remark}

The notion of “totally disconnected” is defined in  \cite{su2007character,chmutov2014weyl} in the regular dominant case as item (1) of the proposition above. Our definition should be viewed as a natural extension to arbitrary weights.
\end{remark}

\begin{proposition}
regular antidominant integral weight is totally disconnected.
\end{proposition}

\begin{proof}
Exactly the same as the proof of \ref{lemmC}, mutatis mutandis.
\end{proof}
\begin{remark}
There exist totally disconnected antidominant weights that are nonregular. Indeed, for \(\mathfrak g=\mathfrak{gl}(m|1)\), every weight is totally disconnected.
\end{remark}